\documentclass[reqno,a4paper,10pt]{amsart}
\usepackage{amsmath, amsthm, amsfonts, amssymb, amscd, cancel, mathrsfs, mathtools, pinlabel, soul, stmaryrd}

\usepackage[T1]{fontenc}
\usepackage[utf8]{inputenc}
\usepackage{palatino}

\usepackage{graphicx}
\usepackage{overpic}
\usepackage{caption}
\usepackage{subcaption}
\usepackage{thmtools}
\usepackage{thm-restate}

\usepackage{todonotes}

\theoremstyle{definition}
\newtheorem{thm}{Theorem}

\newtheorem{prop}[thm]{Proposition}
\newtheorem{cor}[thm]{Corollary}

\newtheorem*{quest*}{Question}

\newtheorem{lemma}[thm]{Lemma}

\newtheorem*{conj*}{Conjecture}
\newtheorem*{probl*}{Problem}

\usepackage{hyperref}
\usepackage{xcolor}
\hypersetup{
	colorlinks=true,
	linkcolor={magenta},
	citecolor={blue},
	urlcolor={black}
}

\usepackage{cleveref}
\usepackage[backend=bibtex,maxbibnames=99,sorting=nty,style=alphabetic]{biblatex}
\renewbibmacro{in:}{%
  \ifentrytype{article}{}{\printtext{\bibstring{in}\intitlepunct}}%
}
\AtBeginBibliography{\footnotesize}
\DeclareFieldFormat{postnote}{#1}
\title{Stein-fillability and positivity in the mapping class group}

\author{Vitalijs Brejevs}
\address{Department of Mathematics, University of Vienna, Vienna, Austria}
\email{vitalijs.brejevs@univie.ac.at}
\author{Andy Wand}
\address{School of Mathematics and Statistics, University of Glasgow, Glasgow, UK}
\email{andy.wand@glasgow.ac.uk}

\begin{document}
\begin{abstract}
    We construct an infinite family of non-positive open books with once-punctured torus pages that support Stein-fillable contact structures. Combined with a result of Wendl, this allows us to give a complete answer to a long-standing question about the mapping class group of a compact surface with boundary: namely, we conclude that the monoid of monodromies supporting Stein-fillable contact structures is equal to the monoid of positive monodromies if and only if the surface is planar.
\end{abstract}
\maketitle
\vspace{-1em}
\section{Introduction}
\label{sec:intro}

A fundamental result of Giroux~\cite{giroux2002} states that for a closed connected oriented 3-manifold $Y$ there exists a one-to-one correspondence
    \begin{equation*}
    \label{eq:giroux}
        \vspace{.1em}
        \left\{ \frac{\textrm{contact structures on } Y}{\textrm{isotopy}}\right\} \longleftrightarrow \left\{ \frac{\textrm{open books for } Y}{\textrm{positive stabilisation}} \right\},
        \vspace{.1em}
    \end{equation*}
which allows one to consider questions in contact and symplectic geometry via the proxy of surface mapping class groups. This approach has proved fruitful, in particular, in the study of fillings of contact manifolds, i.e., symplectic 4-manifolds $(X, \omega)$ that contain the given contact 3-manifold $(Y, \xi)$ as their boundary in a compatible way. Stein fillings are particularly amenable to the view through the lens of mapping class groups since every factorisation into positive Dehn twists of the monodromy of an open book $(\Sigma, \varphi)$ supporting $(Y, \xi)$ yields a Lefschetz fibration that carries a Stein structure and fills $(Y, \xi)$~\cite{eliashberg1990, gompf1998}. Furthermore, it follows from~\cite{giroux2002} as well as results of Loi and Piergallini~\cite{loipiergallini2001}, Akbulut and Özbağcı~\cite{akbulutozbagci2001} and Plamenevskaya~\cite{plamenevskaya2004} that a contact manifold is Stein-fillable if and only if the monodromy of \emph{some} supporting open book admits such a positive factorisation. This prompts a natural question: does this equivalence hold for \emph{every} supporting open book $(\Sigma, \varphi)$, or, in other words, is the implication
\begin{equation*}
    \label{eq:corr}
    \tag{$\dagger$}
        (Y, \xi) \textrm{ is Stein-fillable} \enspace \Longrightarrow \enspace \varphi \textrm{ has a positive factorisation}
\end{equation*}
true for all open books $(\Sigma, \varphi)$ supporting $(Y, \xi)$? 

This question can also be phrased in terms of monoids in the mapping class group $\Gamma_{g,n}$ of a compact oriented surface $\Sigma_{g,n}$ with genus $g \geq 0$ and $n>0$ boundary components: denoting by $\Gamma^{\textrm{\,Stein}}_{g,n} \subset \Gamma_{g,n}$ the monoid of mapping classes appearing as monodromies of open books that support Stein-fillable contact structures\footnote{The monoid structure of $\Gamma^{\textrm{\,Stein}}_{g,n}$ was observed in~\cite{BEvHM2010}.}, and by $\Gamma^+_{g,n} \subset \Gamma^{\textrm{\,Stein}}_{g,n}$ the monoid of classes admitting positive factorisations, we have that these monoids coincide precisely when~\eqref{eq:corr} holds.

Wendl has shown in~\cite{wendl2010} that~\eqref{eq:corr} holds if $\Sigma$ is planar, thus providing a powerful tool that has been of great use in classifying fillings \cite{plamenevskaya2010} as well as shedding light on the support genus of certain contact manifolds \cite{wand2011,ggp2020}. Soon after, however, the second author in~\cite{wand2015}, and, independently, Baker, Etnyre and Van Horn-Morris in~\cite{BEvHM2010}, exhibited Stein-fillable contact manifolds supported by open books with the page $\Sigma_{2,1}$ and non-positive monodromies. In~\cite{brejevswand2023}, the authors improved on this result by constructing infinitely many counterexamples to~\eqref{eq:corr} with the page $\Sigma_{g,n}$ for any pair $(g, n) \neq (1, 1)$ with $g > 0$. The remaining case was studied in~\cite{lisca2014} by Lisca who proved that~\eqref{eq:corr} holds for open books with page $\Sigma_{1,1}$ if the supported $Y$ is a Heegaard Floer $L$-space. Our main result is to exhibit an infinite family of counterexamples to~\eqref{eq:corr} with page $\Sigma_{1,1}$.

\begin{thm}
\label{thm:main}
    Let $ k \geq 0 $. Then $ (\Sigma_{1,1}, \varphi_k) $ with $ \varphi_k = \tau_\delta \tau_\alpha^5 \tau_\beta^{-(15+k)}$, where $\alpha$ and $\beta$ are simple closed curves intersecting exactly once, and $\delta$ a boundary-parallel curve (see Figure~\ref{fig:monodromy}), is an open book supporting a Stein-fillable contact manifold, but $ \varphi_k $ does not admit a positive factorisation.
\end{thm}

\begin{figure}
        \centering
        \includegraphics[width=0.40\textwidth]{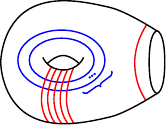}
        \begin{overpic}[width=\linewidth,height=0.001in]{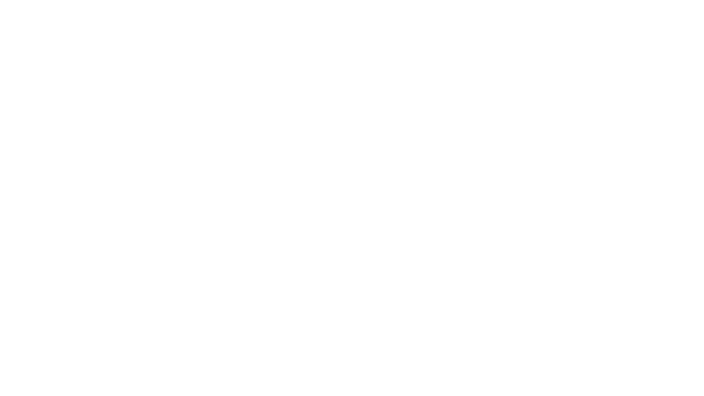}
        \put(50,7) {\rotatebox{35}{$\color{blue} \times (15 + k)$}}
        \put(43,1) {$\alpha$}
        \put(43,27.5) {$\beta$}
        \put(67,7) {$\delta$}
        \end{overpic}
        \caption{
            The open book $ \left(\Sigma_{1,1}, \tau_\delta \tau_\alpha^5 \tau_\beta^{-(15+k)}\right) $ for $ k \geq 0 $, where $\tau_\gamma$ denotes a positive Dehn twist about the curve $\gamma$.
        } \label{fig:monodromy}
    \end{figure}

Theorem~\ref{thm:main} gives an affirmative answer to the question whether the inclusion $\Gamma^+_{1,1} \subset \Gamma^{\textrm{\,Stein}}_{1,1}$ is strict asked by Etnyre and Van Horn-Morris in~\cite[Question~6.2]{etnyrevhm2015}. Moreover, as shown in~\cite[Remark~5.3]{lisca2014} (cf.~\cite[Remark~1.2]{brejevswand2023}), by attaching 1-handles to the page of $(\Sigma_{1,1}, \varphi_k)$ for any $k \geq 0$ and extending the monodromy by the identity on the co-cores of the 1-handles, one can produce a non-positive open book with any non-planar page that supports a Stein-fillable contact structure. Hence, Theorem~\ref{thm:main} in combination with Wendl's result gives a complete answer to~\eqref{eq:corr}:

\begin{cor}
The monoids $\Gamma^+_{g,n}$ and $\Gamma^{\textrm{\,Stein}}_{g,n}$ in the mapping class group of a compact oriented surface with boundary coincide if and only if the surface is planar.
\end{cor}

\subsection*{Acknowledgements} We would like to thank John Etnyre, David Gay and Gordana Mati\'c for illuminating discussions, as well as the anonymous referee for helpful comments. VB was supported by the FWF grant `Cut and Paste Methods in Low Dimensional Topology' and gratefully acknowledges the hospitality of the University of Glasgow, where parts of this work were carried out.

\section{Preliminaries}
\label{sec:prelim}

In this section we briefly recount standard definitions and facts pertaining to our objects of interest in contact topology. We then focus on the group of 3-stranded braids $B_3$ to draw the connection between quasipositive elements of $B_3$ and positive classes in the mapping class group $\Gamma_{1,1}$ of the one-holed torus.

\subsection{Open books, contact manifolds and Stein fillings}

Let $\Sigma_{g,n}$ denote a compact oriented surface of genus $g \geq 0$ with $n \geq 1$ boundary components. Write $ \Gamma_{g,n} $ for the {\emph{mapping class group} of $ \Sigma_{g,n} $ that consists of isotopy classes of orientation-preserving diffeomorphisms $ \Sigma_{g,n} \rightarrow \Sigma_{g,n} $ restricting to the identity on $ \partial \Sigma_{g,n} $; by standard convention, we confuse classes in $\Gamma_{g,n}$ and their representatives. Given $ \varphi \in \Gamma_{g,n} $, say that $ \varphi $ admits a \emph{positive factorisation}, or simply that $\varphi$ is \emph{positive}, if it can be written as a product of positive Dehn twists about essential simple closed curves in $ \Sigma_{g,n} $. We denote by $ \Gamma^+_{g,n} $ the \emph{positive monoid} of $ \Gamma_{g,n} $, i.e., the monoid comprised of isotopy classes of positively factorisable maps.

For our purposes, an \emph{open book} is a pair $(\Sigma, \varphi)$, where $\Sigma = \Sigma_{g,n}$ is called the \emph{page} and $\varphi \in \Gamma_{g,n}$ the \emph{monodromy}. One recovers a closed connected oriented 3-manifold $Y$ from $(\Sigma, \varphi)$ by considering the mapping torus 
    \[
        M_\varphi = [0, 1] \times \Sigma \,/\, (0, \varphi(x)) \sim (1, x)
    \]
and identifying boundary components of $M_\varphi$ with $ \bigsqcup_n S^1 \times S^1 $ such that in each copy of $ S^1 \times S^1 $ the first factor is the quotient of the unit interval and the second comes from a connected component of $ \partial \Sigma $. Then $Y$ is the result of glueing in solid tori $ \bigsqcup_n D^2 \times S^1 $ by the identity map 
    \[
        \bigsqcup_n \partial D^2 \times S^1 \rightarrow \bigsqcup_n S^1 \times S^1.
    \]
Call the link $L \subset Y$ consisting of the cores $ \bigsqcup_n \{ 0 \} \times S^1 $ of the solid tori the \emph{binding} of $(\Sigma, \varphi)$ and identify each page with the quotient of $ \{ \theta \} \times \Sigma$ in $M_\varphi$ for $\theta \in [0,1]$. We say that $Y'$ \emph{admits} an open book $(\Sigma, \varphi)$ if $Y'$ is diffeomorphic to $Y$ constructed as above. Note that conjugating $\varphi$ does not change the diffeomorphism class of $Y$.

Recall that a \emph{(positive) contact structure} on $ Y $ is an oriented plane field $ \xi \subset TY $ such that $ \xi = \ker \alpha $ for some $ \alpha \in \Omega^1(Y) $ satisfying $ \alpha \wedge \mathrm{d} \alpha > 0 $, and call the pair $(Y, \xi)$ a \emph{contact manifold}. Two contact manifolds $(Y, \xi)$ and $(Y', \xi')$ are \emph{contactomorphic} if there exists a diffeomorphism $ Y \rightarrow Y' $ that induces a map carrying $ \xi $ to $ \xi' $. Say that an open book $(\Sigma, \varphi)$ for $Y$ \emph{supports} $ (Y, \xi) $ if $ \alpha $ is positive on the binding and $ \mathrm{d} \alpha $ is positive on the interior of each page. Thurston and Winkelnkemper proved in~\cite{thurstonwinkelnkemper1975} that every open book for $ Y $ supports some $(Y, \xi)$. Furthermore, as noted in the introduction, by the work of Giroux~\cite{giroux2002} there exists a bijection between contact structures $\xi$ on $Y$ modulo contact isotopy and open books $(\Sigma, \varphi)$ that support $ (Y, \xi) $ modulo \emph{positive stabilisation}, i.e., modulo passing to an open book $(\Sigma', \varphi')$, where $\Sigma'$ is obtained from $\Sigma$ by attaching a 1-handle $h$ and $\varphi' = \tau_\alpha \varphi$ for $\tau_\alpha$ a positive Dehn twist about a simple closed curve $\alpha \subset \Sigma'$ that intersects the co-core of $h$ once. This is consistent since the manifold $(Y', \xi')$ supported by $(\Sigma', \varphi')$ is contactomorphic to $(Y, \xi)$.

A \emph{Stein surface} is defined as a complex surface $ W $ equipped with a Morse function $ f : W \rightarrow \mathbb{R} $ such that for every regular point $ c $ of $ f $, the level set $ f^{-1}(c) $ carries a contact structure $ \xi_c $, yielded by the complex tangencies, that induces an orientation of $ f^{-1}(c) $ that agrees with the orientation of $ f^{-1}(c) $ viewed as the boundary of the complex manifold $ f^{-1}((-\infty,c]) $. Given $ (Y, \xi) $, say that it is \emph{Stein-fillable} if $ Y $ is orientation-preserving diffeomorphic to some such $ f^{-1}(c) $ and $ \xi $ is isotopic to $ \xi_c $. For $ (Y, \xi) $ to be Stein-fillable, $ \xi $ must necessarily be \emph{tight}, i.e., there must not exist an embedded disk $ D^2 \subset Y $ that is tangent to $ \xi $ along $ \partial D^2 $; if such a disk exists, $ \xi $ is \emph{overtwisted}. Recall from Section~\ref{sec:intro} that $ (Y, \xi) $ is Stein-fillable if and only if there exists an open book supporting $ (Y, \xi) $ whose monodromy admits a positive factorisation.

\subsection{Quasipositive 3-braids and positive mapping classes}
\label{subsec:3braids}

Given the braid group on $n$ strands $B_n$, we say that $\beta \in B_n$ is \emph{quasipositive} if it can be written as a product of conjugates of the standard generators $\sigma_1, \dots, \sigma_{n-1}$ of $B_n$. Note that quasipositivity is preserved by conjugation and write $QP(n)$ for the monoid of quasipositive $n$-braids; we will be concerned with the case $n = 3$. Recall that $B_3$ admits a  presentation $\left\langle \sigma_1, \sigma_2 \mid \sigma_1 \sigma_2 \sigma_1 = \sigma_2 \sigma_1 \sigma_2 \right\rangle$ and write $\Delta = \sigma_1 \sigma_2 \sigma_1 = \sigma_2 \sigma_1 \sigma_2$ in the following. Also, observe that
    \[
    \label{eq:deltarel}
    \tag{$*$}
        \sigma_1^{\pm 1} \Delta^{\pm 1} = \Delta^{\pm 1} \sigma_2^{\pm 1} \quad \textrm{and} \quad \sigma_2^{\pm 1} \Delta^{\pm 1} = \Delta^{\pm 1} \sigma_1^{\pm 1}.
    \]
    
It is well-known (see, e.g.,~\cite[Chapter~9]{farbmargalit2011}) that there exists an isomorphism $\Psi : B_3 \rightarrow \Gamma_{1,1}$ given by $\Psi(\sigma_1) = \tau_\alpha$ and $\Psi(\sigma_2) = \tau_\beta$, where $\tau_\alpha$ and $\tau_\beta$ are the positive Dehn twists about, respectively, a meridian $\alpha $ and a longitude $\beta$ of $\Sigma_{1,1}$ that intersect transversely once (q.v.~Figure~\ref{fig:monodromy}). Moreover, we have the following folklore result proved by Ito and Kawamuro in~\cite{itokawamuro2019}:

\begin{thm}[{\cite[Theorem~3.2]{itokawamuro2019}}]
    \label{thm:qp}
    The group isomorphism $\Psi : B_3 \rightarrow \Gamma_{1,1}$ restricts to a monoid isomorphism $\Psi|_{QP(3)} : QP(3) \rightarrow \Gamma_{1,1}^+ $.
\end{thm}


An algorithm to determine whether a given 3-braid is quasipositive was first described in~\cite{orevkov2004} and later improved in~\cite{orevkov2015} by Orevkov. Thus, Theorem~\ref{thm:qp} enables us to determine whether any $\varphi \in \Gamma_{1,1}$ is positive by applying the algorithm to the preimage $\Psi^{-1} (\varphi) \in B_3$. The first step is finding a representative $\beta$ of the conjugacy class of $\Psi^{-1} (\varphi)$ in \emph{standard form}, i.e., 
    \[
        \beta = \Delta^p w^+ = \Delta^p \sigma_1^{a_1} \sigma_2^{a_2} \sigma_1^{a_3} \sigma_2^{a_4} \dots,
    \]
where $w^+$ is a word in $m \geq 0$ syllables that alternate between powers of $\sigma_1$ and $\sigma_2$, $a_i > 0$ for all $i = 1, \dots, m$ and $p \in \mathbb{Z}$.\footnote{Note that such $\beta$ always exists since one can write $\sigma_1^{-1} = \Delta^{-1} \sigma_1 \sigma_2 $ and $\sigma_2^{-1} = \Delta^{-1} \sigma_2 \sigma_1$, collect the $\Delta^{-1}$ factors on the left using~\eqref{eq:deltarel}, and conjugate appropriately to ensure that the first syllable in $w^+$ is a power of $\sigma_1$.} If one of the following also holds, say that $\beta$ is in \emph{reduced standard form}:
    \begin{enumerate}
        \item $m \leq 1$;
        \item $m = 2$, $a_1 = a_2 = 1$ and $p \equiv 0 \ \mathrm{mod} \ 2$; or
        \item $m \equiv p \ \mathrm{mod} \ 2$ and $a_i \geq 2$ for all $i = 1, \dots, m$.
    \end{enumerate}

For a complete description of the algorithm, we refer the reader to~\cite[Section~6]{orevkov2015}. However, for the purpose of proving Theorem~\ref{thm:main}, the following quasipositivity obstruction will be sufficient.

\begin{prop}[{\cite[Proposition~6.5(b)]{orevkov2015}}]
\label{prop:nqp}
    Let $\beta$ be given in reduced standard form and denote by $e(\beta)$ the exponent sum of $\beta$, i.e., $e(\beta) = 3p + \sum_{i=1}^m a_i$. Then if $p < 0$ and $\beta$ is quasipositive, we have $0 < p + m < 2 e(\beta)$.
\end{prop}

Since quasipositivity is a property of conjugacy classes, it follows that if $\beta$ in reduced standard form is conjugate to $\Psi^{-1} (\varphi)$ and does not satisfy the inequalities of Proposition~\ref{prop:nqp}, then $\varphi$ is non-positive.

\section{Proof of Theorem~\ref{thm:main}}
\label{sec:proof}

Henceforth, for all $k\geq 0$, let 
    \[
        \beta_k = (\sigma_1 \sigma_2)^6 \sigma_1^5 \sigma_2^{-(15 + k)} \in B_3
    \]
and write $\varphi_k = (\tau_\alpha \tau_\beta)^6 \tau_\alpha^5 \tau_\beta^{-(15+k)} \in \Gamma_{1,1}$ for the image of $\beta_k$ under the isomorphism $\Psi : B_3 \rightarrow \Gamma_{1,1}$ defined in Subsection~\ref{subsec:3braids}. Note that by the chain relation in $\Gamma_{1,1}$ we can write $(\tau_\alpha \tau_\beta)^6 = \tau_\delta$, where $\delta$ is a boundary-parallel curve in $\Sigma_{1,1}$. In the next lemma, we show that the monodromies in the statement of Theorem~\ref{thm:main} are non-positive.

\begin{lemma}
\label{lem:phinp}
    For all $k \geq 0$, we have $\varphi_k \notin \Gamma^+_{1,1}$.
\end{lemma}
\begin{proof}
    Consider $\beta_0 = (\sigma_1 \sigma_2)^6 \sigma_1^5 \sigma_2^{-15} = \Delta^4 \sigma_1^5 \sigma_2^{-15}$. We have
        \begin{align*}
            \beta_0 &= \Delta^4 \sigma_1^5 \sigma_2^{-15} = \Delta^4 \sigma_1^5 \sigma_1 \sigma_2 \Delta^{-1} \sigma_2^{-14} = 
            \Delta^4 \sigma_1^5 \sigma_1 \sigma_2 \sigma_1^{-14} \Delta^{-1} \\
            &= \Delta^4 \sigma_1^5 \sigma_1 \sigma_2 \sigma_2 \sigma_1 \Delta^{-1} \sigma_1^{-13} \Delta^{-1} = \Delta^4 \sigma_1^5 \sigma_1 \sigma_2 \sigma_2 \sigma_1 \sigma_2^{-13} \Delta^{-2} =\dots \\
        \intertext{Applying the identity $\sigma_2^{-k}= \sigma_1 \sigma_2 \sigma_2 \sigma_1 \sigma_2^{-k+2} \Delta^{-2}$ six more times and writing $\sim$ to denote conjugacy in $B_3$, we get}
        \dots &= \Delta^4 \sigma_1^5 (\sigma_1 \sigma_2 \sigma_2 \sigma_1)^7 \sigma_1 \sigma_2 \Delta^{-15} \sim \Delta^{-11} \sigma_1^5 (\sigma_1 \sigma_2 \sigma_2 \sigma_1)^7 \sigma_1 \sigma_2 \\
        & = \Delta^{-11} \sigma_1^6 (\sigma_2^2 \sigma_1^2)^7 \sigma_2 \sim \sigma_2 \Delta^{-11} \sigma_1^6 (\sigma_2^2 \sigma_1^2)^7 = \Delta^{-11} \sigma_1^7 (\sigma_2^2 \sigma_1^2)^7 \eqqcolon \beta_0'.
        \end{align*}

    Observe that $\beta_0'$ is in reduced standard form with $m = 15$ and $p = -11$. Since $e(\beta_0') = 2$, we have $p + m = 4 \nless 4 = 2e(\beta_0')$. By Proposition~\ref{prop:nqp}, $\beta_0'$ is not quasipositive, hence neither is $\beta_0$. Moreover, since every $\beta_k$ for $k \geq 1$ is obtained by adding negative subwords to $\beta_0$, we have $\beta_k \notin QP(3)$ for all $k \geq 0$. By Theorem~\ref{thm:qp} and the discussion in Subsection~\ref{subsec:3braids}, we conclude that $\varphi_k$ is non-positive for all $k \geq 0$.
\end{proof}

\begin{figure}[ht]
        \vspace{-\baselineskip}
        \centering
        \includegraphics[width=\textwidth]{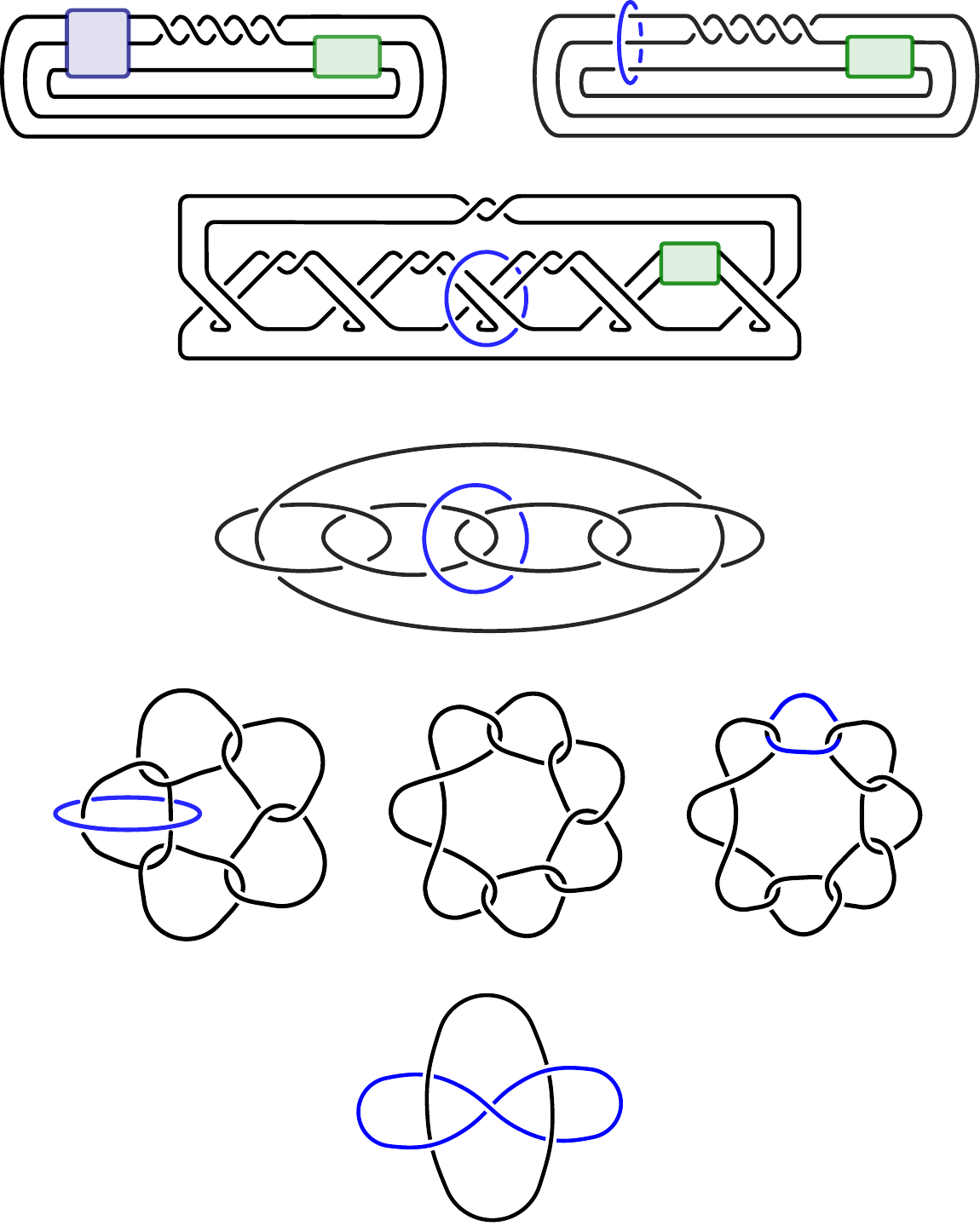}
        \begin{overpic}[width=\linewidth,height=0.001in]{figures/blank.png}

         \put(8.5,123.5) {$\scriptstyle +2$}
         \put(32.5,122.5) {$\scriptscriptstyle -15-k$}

        \put(65.5,128) {$\color{blue}\scriptstyle -1/2$}
        \put(87,122.5) {$\scriptscriptstyle -15-k$}
    
        \put(47,120) {$\xrightarrow[\textrm{up}]{\textrm{blow-}}$}
        \put(83,105){\rotatebox{45}{$\xleftarrow{\textrm{isotopy}}$}}

        \put(47,103.5) {$\color{blue}\scriptstyle -1/2$}
        \put(67.2,101) {$\scriptscriptstyle -17-k$}

        \put(49,85){\rotatebox{90}{$\longrightarrow$}}
        \put(51,86){$2:1$}

        \put(67,81) {$\scriptstyle -2$}
        \put(75,76.5) {$\scriptstyle -17-k$}
        \put(55,78) {$\scriptstyle -2$}
        \put(39,78) {$\scriptstyle -2$}
        \put(22,77) {$\scriptstyle -2$}
        \put(47,80) {$\color{blue}\scriptstyle -1$}

        \put(24,58){\rotatebox{45}{$\xleftarrow{\textrm{isotopy}}$}}
        \put(32.5,45) {$\xrightarrow[\textrm{down}]{\textrm{blow-}}$}
        \put(63.3,45) {$\xrightarrow[\textrm{up}]{\textrm{blow-}}$}
        \put(63,20){\rotatebox{45}{$\xleftarrow[\textrm{and isotopy}]{\textrm{four blow-downs}}$}}

        \put(18,59) {$\scriptstyle -2$}
        \put(28,56) {$\scriptstyle -2$}
        \put(28,34) {$\scriptstyle -2$}
        \put(18,30) {$\scriptstyle -17-k$}
        \put(8,39) {$\scriptstyle -2$}
        \put(5,47.5) {$\color{blue} \scriptstyle -1$}

        \put(53,58.5) {$\scriptstyle -2$}
        \put(60,53) {$\scriptstyle -2$}
        \put(60,36) {$\scriptstyle -2$}
        \put(53,30) {$\scriptstyle -17-k$}
        \put(45,57) {$\scriptstyle -2$}

        \put(81,58.5) {$\color{blue} \scriptstyle +1$}
        \put(87,56) {$\scriptstyle -1$}
        \put(74,56) {$\scriptstyle -1$}
        \put(94.5,45) {$\scriptstyle -2$}
        \put(89,34) {$\scriptstyle -2$}
        \put(81,30) {$\scriptstyle -17-k$}

        \put(39,20) {$\color{blue} \scriptstyle +5$}
        \put(52,2) {$\scriptstyle -15-k$}
        
        \end{overpic}
        \caption{
            $\Sigma_2(\beta_k)$ is diffeomorphic to $\mathrm{Wh}(5, -(15+k))$. Here the blue rectangle is annotated with the number of full twists and the green rectangles with the number of half-twists.\vspace{-\baselineskip}
        } \label{fig:5chain}
    \end{figure}

We will also make use of recent results by Min and Nonino in~\cite{minnonino2023} who classified tight contact structures on a large class of surgeries on the Whitehead link $\mathrm{Wh}$, shown at the bottom of Figure~\ref{fig:5chain}, and determined their Stein fillability. Write $\mathrm{Wh}(r_1, r_2)$ for the $(r_1,r_2)$-surgery on $\mathrm{Wh}$; note that the assignment of the coefficients to the components of $\mathrm{Wh}$ is immaterial as they can be interchanged by an ambient isotopy. Then we have the following:

\begin{thm}[{\cite[Theorem~1.6]{minnonino2023}}]
    \label{thm:sf}
    Let $r_1 \geq 5$ be an integer and $r_2 \in ((-\infty, 0) \cup [2, 4) \cup [5, \infty)) \cap \mathbb{Q}$. Then every tight contact structure on $\mathrm{Wh}(r_1, r_2)$ is Stein-fillable.
\end{thm}

We are now equipped for the proof of our main result.

\begin{proof}[Proof of Theorem~\ref{thm:main}]
    Let $k \geq 0$ and denote by $(Y_k, \xi_k)$ the contact manifold supported by the open book $(\Sigma_{1,1}, \varphi_k)$. Theorem~4.2 in~\cite{baldwin2008} implies that $(Y_k, \xi_k)$ is tight since the Ozsv\'ath--Szab\'o contact invariant $c(\xi_k) \neq 0$. 
    
    We now show that $Y_k$ is diffeomorphic to $\mathrm{Wh}(5, -(15 + k))$, whence Stein-fillability of $\xi_k$ will follow by Theorem~\ref{thm:sf}. To this end, we note firstly the standard fact that $Y_k$ is diffeomorphic to $\Sigma_2({\beta}_k)$, the double cover of $S^3$ branched over the closure of $\beta_k$ (see, e.g.,~\cite[Lemma~2.1]{baldwin2008}). We then follow the strategy of~\cite[Section~1]{simone2023} (cf.~\cite[Lemma~2.2]{brejevssimone2024}) illustrated in Figure~\ref{fig:5chain}: we may blow up the closure of $\beta_k$ with a $-1/2$-framed unknot and isotop the resulting diagram to make it apparent via standard Akbulut--Kirby techniques~\cite{akbulutkirby1980} that $\Sigma_2(\beta_k)$ is given by the surgery on a 5-chain link with four components framed $-2$ and the fifth $-(17+k)$, and a $-1$-framed unknot threaded through it; this also follows directly from~\cite{simone2023} by writing
    \[
        \beta_k = (\sigma_1 \sigma_2)^6 \left(\sigma_1 \sigma_2^{-(2-2)}\right)^4 \sigma_1 \sigma_2^{-(17 + k - 2)}.
    \]
    By blowing down said unknot, then blowing up the result by a $+1$-framed unknot and performing four further blow-downs on the $-1$-framed components, we obtain a surgery diagram for $\mathrm{Wh}(5, -(15+k))$, as desired. 
    
    Since the monodromy $\varphi_k$ is non-positive by Lemma~\ref{lem:phinp}, the result follows.\qedhere
\end{proof}


\printbibliography

\end{document}